\documentclass[11pt,a4paper]{article}
\title{Finite-pool queues with heavy-tailed services}
\date{\today}
\author{Gianmarco Bet, Remco van der Hofstad, and Johan S.H. van Leeuwaarden} 

\usepackage{mathrsfs} 
\usepackage{standalone} 
\usepackage{pgfplots}
\usepackage{caption}    %
\usepackage{subcaption} 
\usepackage{gianmarcos_definitions}

\DeclareRobustCommand{\VAN}[3]{#2} 

\begin{document}
\maketitle
\begin{abstract}
We consider the $\Delta_{(i)}/G/1$ queue, in which a a total of $n$ customers independently demand service after an exponential time. We focus on the case of \emph{heavy-tailed} service times, and assume that the tail of the service time distribution decays like $x^{-\alpha}$, with $\alpha\in(1,2)$. We consider the asymptotic regime in which the population size grows to infinity and establish that the scaled queue length process converges  to an $\alpha$-stable process with a negative quadratic drift. We leverage this asymptotic result to characterize the headstart that is needed to create a long period of activity. This result should be contrasted with the case of light-tailed service times, which was shown to have a similar scaling limit \cite{bet2014heavy}, but then with a Brownian motion instead of an $\alpha$-stable process.
\end{abstract}
\section{Introduction}
This paper deals with the $\DG$ queue, designed to model a service system to which only a finite pool of customers can arrive. The $n$ potential customers in the pool each have an i.i.d.~exponential clock and join the queue when their clock rings. Each customer joins the queue only once, and at the system level, this creates an arrival process governed by the order statistics of the clock times. The $\DG$ queue, in the precise form as studied in this paper, was introduced in \cite{honnappa2015delta, honnappa2014transitory}, and belongs to the branch of queueing theory that deals with time-dependent or transient conditions \cite{mandelbaum1995strong, massey1982non, massey1985asymptotic, newell1968queuesIII}. Indeed, with time, the pool of potential customers (those who have not joined yet) becomes smaller, and the influx of customers loses intensity. 

The $\DG$ queue can be studied in many operating regimes; see e.g.~\cite{honnappa2015delta, louchard1994large}, where the focus lies on overloaded regimes and see \cite{bet2014heavy} for a detailed overview.  In \cite{bet2014heavy} we introduced a new heavy-traffic regime defined by two features: The customer pool grows to infinity and the initial (at time zero) rate of newly arriving customers is such that, on average, one new customer is expected to arrive during one service time. This gives rise to a large-scale system that (initially) operates close to full utilization, and is expected to utilize its resources efficiently. By this we mean that the server is typically busy, and that idle times are negligible. In fact, that is the main goal of this paper: To characterize the conditions under which sufficiently many customers will join the queue to guarantee that the system will have a substantial backlog of customers. We therefore focus on the first busy period, and show how to set the initial number of customers already present in the queue at time $t=0$, referred to as the \emph{head start}, to create a considerable first busy period during which the server can work continuously.

In \cite{bet2014heavy} we have studied this heavy-traffic regime under the assumption that the second moment of the service time distribution is finite, and showed that the queue length process converges to a reflected Brownian motion with negative quadratic drift. The negative drift captures the effect of a  pool of potential customers that diminishes with time: After some time, the activity in the queue inevitably becomes negligible. However, in the early phases the rate of arriving customers can be high. Say you want to start a business and you estimate that a population of $n$ persons might become customers. Then the head start can be interpreted as the persons that signed up (already) as a customer. In \cite{bet2014heavy} we showed that in our heavy-traffic regime, once the head start is of order $n^{1/3}$, and you would decide to start your business, the number of customers you will serve consecutively is of the order $n^{2/3}$. With this mental picture, our heavy-traffic regime gives insight into dimensioning rules about how large the pool $n$ should be in comparison to the head start, how to choose the service capacity as a function of the pool size to achieve full system utilization, and how to control the first busy period, which essentially is the relevant time of operation of the system.

In the present paper we drop the finite second moment condition and study the queue length process under the additional assumption that the service times are {\it heavy tailed}. More precisely, we assume that the service times follow a power-law distribution with power-law exponent $\alpha\in(1,2)$. We establish that in a similar heavy-traffic regime as in \cite{bet2014heavy} the rescaled queue length process converges to an $\alpha$-stable process with negative quadratic drift. As in the finite second moment case, the diminishing pool effect is still there in the form of the drift term, but the oscillations of the limiting queue length are much wilder. We will also show that, as a consequence of the larger fluctuations, the desired head start and canonical busy period should scale with $n$, in a specific way that vitally depends on the exponent $\alpha$ and other more refined properties of the service time distribution.

\section{Description of the model}
We now describe in detail the $\Delta_{(i)}/G/1$ model with exponential arrivals. We consider a population of $n$ potential customers that are to be served by a single server. Each customer $i\in\{1,\ldots,n\}=: [n]$ is assigned a random variable $T_i$, representing its arrival time. We assume $(T_i)_{i=1}^n$ to be a sequence of i.i.d.~exponential random variables with mean $1/\lambda$. We denote the distribution function of  $T_1$ by $\FT(\cdot)$. When the clock $T_i$ rings, customer $i$ joins the queue and customers in the queue are served in a first-come-first-served manner. The arrival times are then given by the order statistics of $(T_i)_{i=1}^n$. The service times are given by a sequence of independent random variables $(S_i)_{i=1}^n$. We denote the distribution function of  $S_1$ by $F_{\sss S}(\cdot)$. Recall that a function $\ell(\cdot)$ is said to be \emph{slowly varying} when $\lim_{t\rightarrow\infty}\ell(tc)/\ell(t) = 1$ for all $c>0$.  The service time distribution is assumed to be in the domain of attraction of an $\alpha$-stable law i.e., its tail decays as
\begin{equation}\label{eq:definition_alpha_stable_service_distribution}%
\mathbb P(S > t) = 1 - F_{\sss S} (t) =  t^{-\alpha}\ell (t),\qquad \alpha\in(1,2),
\end{equation}%
for a slowly-varying function $\ell(\cdot)$. Assumption \eqref{eq:definition_alpha_stable_service_distribution} implies, in particular, that $\mathbb{E}[S^{k}] = \infty$ for $k > \alpha$, and $\mathbb E[S^k]< \infty$ for $k < \alpha$. We further assume that the queue obeys the heavy-traffic condition
\begin{equation}%
\max_{t\geq0}\fT(t)\mathbb E[S] = 1,
\end{equation}%
where $\fT(\cdot)$ is the density of the arrival time distribution. For  exponential arrival times with rate $\lambda$, this condition simplifies to 
\begin{equation}\label{eq:definition_criticality_condition}%
\lambda\mathbb E[S] = 1,
\end{equation}%
which can be interpreted as follows. Since $\lambda$ represents the instantaneous (close to time $t = 0$) arrival rate of customers, \eqref{eq:definition_criticality_condition} amounts to assuming that on average, during one service time, one customer joins the queue.
We study the queue after $ N_n(0)$ customers have already joined, where $N_n(0)$ may depend on $n$. Without loss of generality, we can assume there are (still) $n$ customers in the pool. Before stating our main results, let us introduce some notation. 

Denote the number of arrived customers in the interval $[0,t]$ by
\begin{equation}\label{eq:delta_arrival_process}%
A_n(t) = \sum_{i=1}^n \mathds 1_{\{T_i\leq t\}}.
\end{equation}%
Let 
\begin{equation}%
\sigma (t) = \max\Big\{k\geq0\Big\vert \sum_{i=1}^k S_i \leq t\Big\}
\end{equation}%
be the renewal process associated with the service times $(S_i)_{i=1}^n$ and define the net input process as 
\begin{equation}\label{eq:net_input_process}%
X_n(t) = \sum_{i=1}^{A_n(t)}S_i -t.
\end{equation}%
The process $X_n(\cdot)$ is useful in defining the cumulative busy time process as
\begin{equation}%
B_n(t) = t - I_n(t) = t - \inf_{0\leq s\leq t}(X_n(s)^-),
\end{equation}%
where $f(x)^-=\min\{0,f(x)\}$  (resp.~$f(x)^+=\max\{0,f(x)\}$) and $I_n(\cdot)$ is the cumulative idle time.

Let $\mathcal D := \mathcal D ([0,\infty))$ denote the space of c\`adl\`ag functions that  are continuous from the right and admit a limit from the left at every point. All the functions that we consider are  elements of $\mathcal D$. Let $\phi(\cdot):\mathcal D\mapsto\mathcal D$ be the \emph{reflection mapping}, defined as
\begin{equation}\label{eq:definition_reflection_mapping}%
\phi(f)(x) = f(x)+\psi(f)(x), 
\end{equation}%
where $\psi(\cdot):\mathcal D\mapsto\mathcal D$ is given by
\begin{equation}\label{eq:definition_local_time}%
\psi(f)(x) = - \inf_{0\leq y\leq x}(f(y)^-).
\end{equation}%
We define the queue length process $Q_n(t)$ by
\begin{equation}\label{eq:real_definition_queue_length}%
Q_n(t) = N_n(0) +  A_n(t) -  \sigma(B_n(t)),
\end{equation}%
where $N_n(0)$ denotes the number of customers already in the queue at the beginning of the first service. 
The time change $t\mapsto B_n(t)$ depends both on $(T_i)_{i=1}^{\infty}$ and $(S_i)_{i=1}^{\infty}$ and as such makes the analysis of $Q_n(t)$ challenging. One popular approach, pioneered by Iglehart and Whitt \cite{iglehart1970multiple}, consists in studying a related queue in which the server never idles, but rather continues working according to the renewal process associated with $(S_i)_{i=1}^{\infty}$ even when the queue is empty. This is often referred to as the queue with \emph{autonomous service} or the Borovkov modified system, see \cite[Chapter 10.2]{StochasticProcess}. It turns out that, under mild assumptions, the original queue and the Borovkov modified system are asymptotically equivalent in heavy traffic \cite[Theorem 10.2.2]{StochasticProcess}, in the sense that the distance between the two queue length processes converges to zero. However, for this approach to work, the service time limit process needs to be continuous. Indeed, the distance between the two processes is bounded from above by the (scaled) maximum service time, or equivalently the maximum jump functional applied to the service time process. If the service time limit process is continuous, the maximum jump functional converges to zero. If, on the other hand, the service time limit process is discontinuous, the distance between the two queues cannot be shown to converge to zero.

Instead, here we adopt a different approach that allows us to deal with a discontinuous service time limit process. This consists in expressing $Q_n(\cdot)$ as the reflection of an appropriate free process $N_n(\cdot)$. Since, after rescaling, $N_n(\cdot)$ converges and the reflection mapping is continuous a.s.~in the limit point, by the Continuous Mapping Theorem the process $Q_n(\cdot)$ also converges. The free process $N_n(\cdot)$ has the following interpretation: When the server is working, $N_n(\cdot)$ follows $Q_n(\cdot)$. When the queue is empty, $N_n(\cdot)$ decreases linearly at a rate equal to the service rate. Therefore, while in the Borovkov modified system the server works continuously according to the service time renewal process, in the  process $N_n(\cdot)$ when there are no customers in the system, the server provides instantaneous work with rate $1/\E[S]$. Consequently, the process $N_n(\cdot)$ can be seen as a fluid version of the Borovkov modified system. In Figure \ref{fig:free_process} we plot a sample path of the process $N_n(\cdot)$. The process $Q_n(\cdot)$ can then be represented as follows:
\begin{lemma}[Alternative formulation for the queue length]%
The queue length process $(Q_n(t))_{t\geq0}$ can be represented as
\begin{equation}\label{eq:definition_queue_length}%
Q_n(t) = \phi (N_n)(t),\qquad t\geq0,
\end{equation}%
where  $N_n(\cdot)$ is given by
\begin{equation}\label{eq:definition_pre_reflection_process}%
N_n(t) = N_n(0) +  A_n(t) -  \sigma(B_n(t)) - (t-B_n(t))/\E[S].
\end{equation}%
\end{lemma}%
\begin{proof}
Start from \eqref{eq:real_definition_queue_length}, and add and subtract $(t-B_n(t))/\E[S]$, to obtain
\begin{align}%
Q_n(t) &= N_n(0) +  A_n(t) -  \sigma(B_n(t)) - (t-B_n(t))/\E[S] + (t-B_n(t))/\E[S]
\end{align}%
The representation \eqref{eq:definition_queue_length} will follow as the solution of the so-called \emph{Skorokhod problem}:
\begin{lemma}[Skorokhod problem {\cite[Proposition 2.2, p.~251]{asmussen2003applied}}]\label{lem:skorokhod_problem}%
Let $F(\cdot)$ be a real-valued stochastic process such that $F(0) = 0$. Assume $R(\cdot)$ is a non-decreasing right-continuous process such that the process $Q(\cdot)$ given by $Q(0) = q$ and $Q(t) = F(t) + R(t)$ is nonnegative for all $t$, and $\int_{0}^{\infty}Q(t)\mathrm d R(t) = 0$. Then $R(t) = \psi(q+F(\cdot))(t)$ and $Q(t) = \phi(q+F(\cdot))(t)$ (recall \eqref{eq:definition_reflection_mapping} and \eqref{eq:definition_local_time}).
\end{lemma}%
We now apply Lemma \ref{lem:skorokhod_problem} with the choices 
\begin{align}
F(t)&= A_n(t) -  \sigma(B_n(t)) - (t-B_n(t))/\E[S],\nnl
R(t) &=  N_n(0) + (t-B_n(t))/\E[S],\nnl
Q(t) &= Q_n(t).
\end{align}%
Note that $R(t) = N_n(0) + I_n(t)/\E[S]$, so that $R(t)$ is non-decreasing. This, together with the fact that $R(t)\geq0$, implies that $R(t)$ is of bounded variation. Moreover, $R(t)=N_n(0) + I_n(t)/\E[S]$ increases if and only if $Q_n(t) = 0$, that is $\int_0^{\infty}Q(t)\mathrm d R(t)=0$. Then Lemma \ref{lem:skorokhod_problem} implies that 
\begin{equation}%
R(t) = -\inf_{0\leq s \leq t} (N_n(0) +  F(t))^- = \psi(N_n)(t),
\end{equation}%
and
\begin{equation}%
Q(t) =  N_n(0) + F(t) -\inf_{0\leq s \leq t} (N_n(0) +  F(t))^- = \phi(N_n)(t).
\end{equation}%
\end{proof}
\begin{figure}[!hbt]%
\centering
\includegraphics{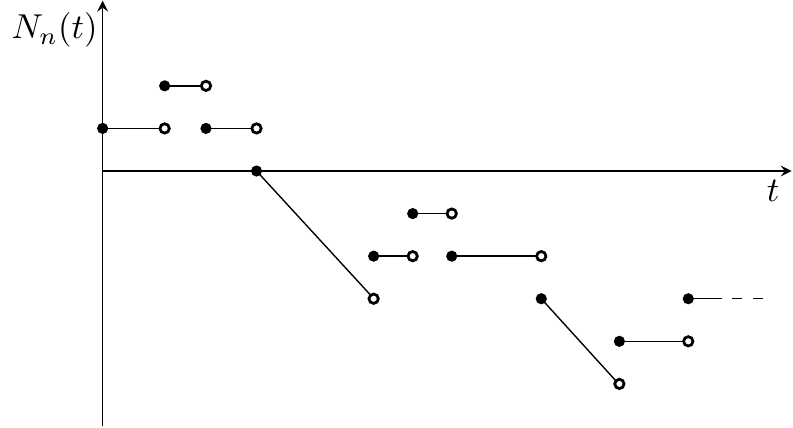}
\caption{A sample path of the process $N_n(\cdot)$.}\label{fig:free_process}
\end{figure}%
%

We will consider the scaled versions of the processes of interest
\begin{align}\label{eq:rescaled_processes}%
\mathbf N_n(t) &= n^{-\frac{1}{2\alpha-1}}\ell_2(n)N_n(\tau_n(t)),\\
\mathbf Q_n(t) &= \phi (\mathbf N_n)(t),\\
\tau_n(t) &= tn^{\frac{\alpha}{2\alpha-1}}\ell_1(n),
\end{align}%
where $\ell_1(\cdot)$ and $\ell_2(\cdot)$ are slowly varying functions that depend on $\ell(\cdot)$ in \eqref{eq:definition_alpha_stable_service_distribution} and are given explicitly in \eqref{eq:scaling_constant_time} and \eqref{eq:scaling_constant_space} below.
We can now state our main result.
\begin{theorem}[Scaling limit for the queue length process]\label{th:scaling_limit_queue_length}
Assume $N_n(0) = q_0 n^{\frac{1}{2\alpha-1}}\ell_2^{-1}(n)$ for some $q_0\geq0$. Then
\begin{equation}\label{eq:pre_reflection_pre_limit_queue_length}%
\mathbf N_n(\cdot) \sr{\mathrm{d}}{\rightarrow} \mathcal N(\cdot)\qquad\mathrm{in}~(\mathcal D, M_1),
\end{equation}%
where
\begin{equation}\label{eq:pre_reflection_scaling_limit_queue_length}
\mathcal N(t) = q_0 + s_{\alpha}\mathcal{S}(t) -\frac{\lambda^2}{2}t^2,
\end{equation}
$s_{\alpha} = \frac{1}{\E[S]^{1+1/\alpha}}$ and $\mathcal{S}(\cdot)$ is a spectrally positive $\alpha$-stable process. Moreover,
\begin{equation}\label{eq:scaling_limit_queue_length}%
\mathbf Q_n(\cdot)\sr{\mathrm{d}}{\rightarrow} \phi(\mathcal{N})(\cdot)\qquad \mathrm{in}~(\mathcal D,M_1).
\end{equation}%
\end{theorem}
Using basic properties of slowly-varying functions (see e.g.~\cite[Proposition 1.3.6]{bingham1989regular}), the scaling constants can be rewritten  as
\begin{align}%
 n^{\frac{\alpha}{2\alpha-1}}\ell_1(n)=n^{\frac{(1+o(1))\alpha}{2\alpha-1}},
\quad n^{-\frac{1}{2\alpha-1}}\ell_2(n) = n^{-\frac{1+o(1)}{2\alpha-1}}.
\end{align}%
Convergence in  $(\mathcal D, M_1)$ is a shorthand notation for convergence in distribution in the space of c\`adl\`ag functions $\mathcal D$ endowed with the $M_1$ topology. We elaborate on this in Section \ref{sec:notation}. In particular, for $\alpha = 2$ the scaling exponents are (asymptotically) the same as in the finite second moment case (see \cite{bet2014heavy}). In Figure \ref{fig:stable_motion_examples} we plot some sample paths of $\phi(\mathcal N)(\cdot)$ for different choices of $\alpha$. As expected, as $\alpha$ approaches $2$, the reflected stable motion starts to resemble a reflected Brownian motion. Notice also that by appropriately choosing $q_0>0$ it is possible to obtain a sizeable first busy period. 
\begin{figure}[!hbt]
	\centering
	
	\begin{subfigure}[b]{0.41\textwidth}
	\includegraphics{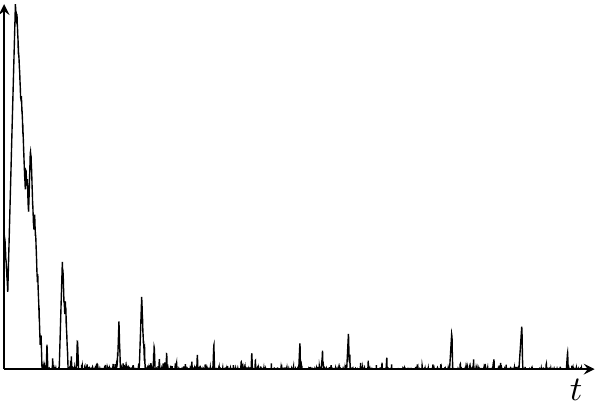}
	
	\caption*{$\alpha = 1.1$}
	\end{subfigure}%
	\begin{subfigure}[b]{0.35\textwidth}
	\includegraphics{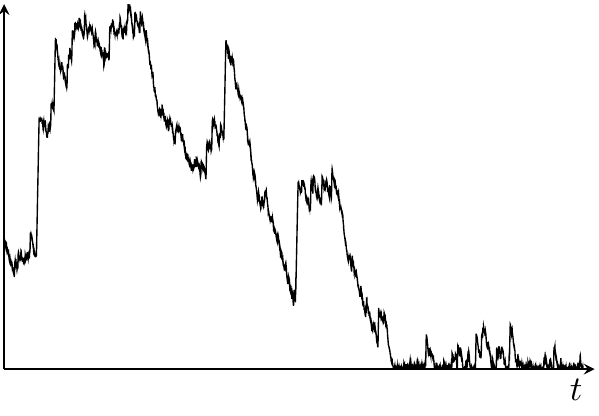}%
	
	\caption*{$\alpha = 1.5$}
	\end{subfigure}
	
	\medskip
	
	\begin{subfigure}[b]{0.41\textwidth}
	\includegraphics{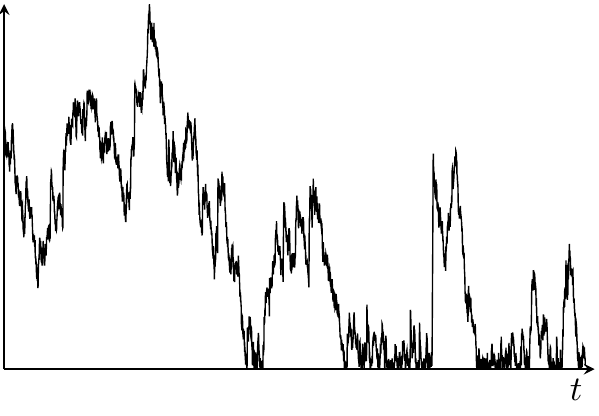}%
	
	\caption*{$\alpha = 1.9$}
	\end{subfigure}%
	\begin{subfigure}[b]{0.35\textwidth}
	\includegraphics{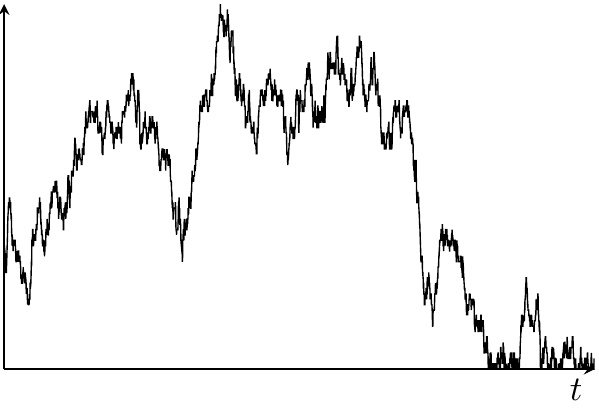}%
	
	\caption*{$\alpha = 2$}
	\end{subfigure}
	\caption{Sample paths of $\phi(\mathcal N)(\cdot)$ for different choices of the power-law exponent $\alpha\in(1,2)$. The case $\alpha = 2$ corresponds to the finite second moment case \cite{bet2014heavy}. In all the cases $q_0=0.1$ and $\lambda = s_{\alpha} = 1$.}	
	\label{fig:stable_motion_examples}
	
\end{figure}%

Define the first hitting time for $\mathbf N_n(\cdot)$ of $x$ as
\begin{equation}%
H_{\mathbf N_n}(x) = \inf\{t > 0: \mathbf N_n(t) \leq x\}.
\end{equation}%
Then $H_{\mathbf N_n}(0)$ represents the first busy period of the $\Delta_{(i)}/G/1$ queue. The following corollary of Theorem \ref{th:scaling_limit_queue_length} characterizes the limiting distribution of $H_{\mathbf N_n}(0)$:
\begin{corollary}[Busy period convergence]
Under the assumptions of Theorem \ref{th:scaling_limit_queue_length}, as $n\rightarrow\infty$,
\begin{equation}%
H_{\mathbf N_n}(0) \sr{\mathrm d}{\rightarrow} H_{\mathcal N}(0)
\end{equation}%
\end{corollary}
\begin{proof}
Note that $t\mapsto\mathcal N(t)$ only has positive jumps. Then, by \cite[Chapter VI, Proposition 2.11]{jacod2003limit}, the functional $f\mapsto H_f(0)$ is continuous in $\mathcal N$ with probability one when $\mathcal D$ is endowed with the $M_1$ topology (indeed, it is continuous when $\mathcal D$ is endowed with the stronger $J_1$ topology). The conclusion follows by an application of the Continuous-Mapping Theorem.
\end{proof}
%

\subsection{Outline}
In Section \ref{sec:preliminaries} we give some background on $\alpha$-stable laws and a derivation of the scaling constants in \eqref{eq:rescaled_processes}. In Section \ref{sec:proof_main_theorem} we provide the proof of Theorem \ref{th:scaling_limit_queue_length}. We prove, separately, the convergence to the stable motion, which follows from classical arguments, and the convergence to the parabolic drift, which is done through a novel coupling argument.
Finally, in Section \ref{sec:discussion} we present some conclusions and open problems.
\section{Preliminaries}\label{sec:preliminaries}
\subsection{Notation} \label{sec:notation}
A common topology on $\mathcal D$, which extends the uniform topology $U$, is the so-called $J_1$ topology, defined by Skorokhod in \cite{skorokhod1956limit}. However, when dealing with limit processes with \emph{unmatched jumps} the coarser $M_1$ topology is needed. When dealing with vector-valued functions (taking values, say, in $\mathbb R^k$)we make use of the \emph{weak} $M_1$ topology $M^{\mathrm{W}}_1$, which coincides with the product topology on $\mathcal D\times \mathcal D\times\cdots \times \mathcal D = \mathcal D^k$. For an in-depth discussion on the various Skorokhod topologies, see \cite{StochasticProcess}. For most results in this paper, convergence of processes means convergence in distribution in the space $\mathcal D$ endowed with the $M_1$ topology. Recall that convergence $x_i\dconv x$ in $\mathcal D([0,\infty))$ is equivalent to convergence in $\mathcal D([0,T])$ for all $T$ that are continuity points of $x$. For a sequence of random variables $(X_i)_{i=1}^{\infty}$ we say that $X_i$ converges to $X$ in probability, and denote it by $X_i\Pconv X$, if for each $\varepsilon>0$, $\mathbb P(\vert X_i - X\vert \geq \varepsilon)\rightarrow0$ as $i\rightarrow\infty$.
With $X_i = \oP(Y_i)$ we mean that $X_i/Y_i\Pconv0$ as $i\rightarrow\infty$. Given two functions $f(\cdot), g(\cdot)$(either on the real numbers or on the integers)  the notation $f\sim g$ means $\lim_{x\rightarrow\infty} f(x)/g(x) = 1$, where $x\in\mathbb R$ or $x\in\mathbb N$. 
%
\subsection{FCLT}
We will establish a FCLT for $\sigma(\cdot)$, as this will give the stable motion in the limit. To do so we exploit the well-known equivalence between FCLT for partial sums and counting processes. 
Let $(S_i)_{i=1}^n$ be a sequence of nonnegative random variables and let 
\begin{equation}\label{eq:FCLT_equivalence_partial_sum_definition}%
\bm{\Sigma}_n(t) = \frac{\Sigma_{\lfloor nt\rfloor}-\E[S]nt}{c_n},
\end{equation}%
be its rescaled partial sum, where $\Sigma_k = S_1+\cdots+S_k$ and $(c_n)_{n=1}^{\infty}$ will be chosen appropriately. 
. Moreover, let $\sigma(t) = \max\{k\geq0\vert \Sigma_k\leq t\}$ and $\bm{\sigma}_n(\cdot)$ the corresponding rescaled process
\begin{equation}\label{eq:FCLT_equivalence_renewal_definition}%
\bm{\sigma}_n(t) = \frac{\sigma(nt) - \E[S]^{-1}nt}{c_n}.
\end{equation}%

The relation between the scaling limits of $\bm{\Sigma}_n(\cdot)$ and $\bm{\sigma}_n(\cdot)$ is described in the following theorem:
\begin{theorem}[{FCLT equivalence \cite[Theorem 7.3.2]{StochasticProcess}}]\label{th:FCLT_equivalence}%
Assume $(S_i)_{i\geq1}$ is a sequence of non-negative random variables, and  $(c_n)_{n\geq1}$ is such that $c_n\rightarrow\infty$, $n/c_n\rightarrow\infty$. Then,
\begin{equation}\label{eq:FCLT_equivalence_partial_sum_convergence_to_stable}%
\bm{\Sigma}_n(\cdot)\sr{\mathrm d}{\rightarrow} \mathcal S(\cdot)\qquad \mathrm{in}~(\mathcal D,M_1)
\end{equation}%
for some process $\mathcal S(\cdot)$
if and only if 
\begin{equation}\label{eq:FCLT_equivalence_renewal_convergence_to_stable}%
\bm{\sigma}_n(\cdot) \sr{\mathrm d}{\rightarrow} - \E[S]^{-1} \mathcal S\circ \E[S]^{-1}\emph{id}(\cdot)\qquad \mathrm{in}~(\mathcal D,M_1),
\end{equation}%
where $\emph{id}(\cdot)$ is the identity function.
\end{theorem}%
The topology $M_1$ plays a crucial role in Theorem \ref{th:FCLT_equivalence}. Indeed, it can be seen that while \eqref{eq:FCLT_equivalence_partial_sum_convergence_to_stable} holds in most cases in the $J_1$ topology, the convergence \eqref{eq:FCLT_equivalence_renewal_convergence_to_stable} can only take place in the $M_1$ topology when the limit process has positive jumps, see \cite[Chapter 7.3.2]{StochasticProcess} for a discussion on this fact. By assumption \eqref{eq:definition_alpha_stable_service_distribution} the sequence $(S_i)_{i\geq1}$ is in the domain of attraction of an $\alpha$-stable motion, that is \eqref{eq:FCLT_equivalence_partial_sum_convergence_to_stable} holds, and
$\mathcal S(\cdot)$ is a centered, spectrally positive $\alpha$-stable motion. 
Then, by Theorem \ref{th:FCLT_equivalence} the process $\bm{\sigma}_n(\cdot)$ is also in the domain of attraction of an $\alpha$-stable motion. Note that the space scaling constants $c_n$ in \eqref{eq:FCLT_equivalence_partial_sum_definition} and \eqref{eq:FCLT_equivalence_renewal_definition} are the same.
\subsection{Poissonian representation of the arrival process}
In order to  further simplify the representation of $Q_n(t)$ we now introduce an alternative characterization of the arrival process as a thinned, marked Poisson process. It is constructed as follows. Given $(\Pi(t))_{t\geq0}$, a homogeneous Poisson process with parameter $\lambda n$, assign to each of its points a mark chosen uniformly in $[n]:=\{1,\ldots, n\}$. We then discard a point if it has a mark that has already been observed in the past. Therefore, conditioned on the (different) marks $M_1,\ldots, M_{k-1}$, the next point of $(\Pi(t))_{t\geq0}$ will be accepted with probability $(n-\vert \{M_1,\ldots, M_{k-1}\}\vert)/n = 1 - (k-1)/n$. We denote this thinned process as $A_n^m(t)$. Then, $A_n^m(t)$ can be represented as
\begin{equation}\label{eq:definition_arrival_process}%
A_n^m(t) = \Pi(t) - R_n(t),
\end{equation}%
where $R_n(t)$ counts the number of \emph{repeated} marks. We emphasize that $\Pi(\cdot)$ and $R_n(\cdot)$ are \emph{not} independent. The arrival process just defined is strongly related with the i.i.d.~sampling in the $\Delta_{(i)}/G/1$ queue, as we now discuss. 

In the $\DG$ queue arrivals are given by an i.i.d.~sequence of arrival clocks $(T_{i})_{i=1}^n$, where it is assumed that customer $i\in[n]$ joins the queue at time $T_i$. This definition departs from the usual queueing assumption of a renewal process, which entails i.i.d.~\emph{inter}-arrival times rather than \emph{arrival} times. However, the above characterization as a marked Poisson process, which holds when the arrival clocks are exponentially distributed, is closer to the usual renewal setting. In what follows we show that the two are equivalent. First, let us introduce some preliminary notation and results.
Given a sequence of random variables $(X_i)_{i=1}^n$, let $X_{(1)}\leq X_{(2)}\leq\cdots\leq X_{(n)}$ denote their order statistics. When $(X_i)_{i=1}^n$ are i.i.d.~exponential random variables, the distribution of the order statistics are well known:
\begin{lemma}[Order statistics of exponentials]\label{order_statistics}
Let $E_1,\ldots,E_n$ be independent exponentially distributed random variables with mean one. Then,
\begin{align}%
(E_{(j)})_{j=1}^n\stackrel{d}{=}\Big(\sum_{s=1}^j\frac{E_s}{n-s+1}\Big)_{j=1}^n.
\end{align}%
\end{lemma}
See for example \cite[Section 2.5]{david2003order} for a proof. Lemma \ref{order_statistics} allows us to relate the process $A_n^m(\cdot)$ we just defined to the arrival process in the $\DG$ queue, as follows:
\begin{lemma}%
For all $t\geq0$,
\begin{equation}%
A_n^m(t) \sr{\mathrm d}{=} A_n(t).
\end{equation}%
\end{lemma}%
\begin{proof}

The (ordered) arrival times in the $\Delta_{(i)}/G/1$ queue are precisely the order statistics of $(T_{i})_{i=1}^n$ and the interarrival times are the  differences between the order statistics. By Lemma \ref{order_statistics}, the distributions of the interarrival times are
\begin{equation}%
E_{(k)}-E_{(k-1)} \sr{\mathrm d}{=}\frac{E_k}{n-k+1},\qquad k\geq1,
\end{equation}%
where we set $E_{(0)} = 0$ for convenience. Multiplying both sides by $n$ gives
\begin{equation}\label{eq:difference_order_statistics_exponentials}%
n(E_{(k)}-E_{(k-1)}) \sr{\mathrm d}{=}\frac{E_k}{1-\frac{k-1}{n}}.
\end{equation}%
Now consider the arrival process  defined in \eqref{eq:definition_arrival_process}. Conditioned on the process up to the arrival $k-1$, the next point of $\Pi(\cdot)$ is accepted with probability $1-\frac{k-1}{n}$. Then, since $\Pi(\cdot)$ is a rate one Poisson process, the time at which the next point of $A^m_n(\cdot)$ occurs is distributed as an exponential random variable with rate $1\cdot(1-\frac{k-1}{n})$.
Equation \eqref{eq:difference_order_statistics_exponentials} then implies that the inter-arrival times in the arrival process just defined are equal (in distribution) to the inter-arrival times of \eqref{eq:delta_arrival_process}.
\end{proof}

\subsection{Determining the scaling constants}\label{sec:determining_scaling_constants}
We now derive the space and time scaling that allow us to obtain the limit process $\mathcal N(\cdot)$ in \eqref{eq:pre_reflection_scaling_limit_queue_length}. We first derive the scaling of time denoted by $k = k(n)$. It is well known that, whenever the limit $\mathcal S(\cdot)$ in \eqref{eq:FCLT_equivalence_partial_sum_convergence_to_stable} is an $\alpha$-stable motion, the fluctuations of $\sum^{\lfloor kt\rfloor}_{i=1}S_i$ around its mean are of the order $c_k=\ell_0(k) k^{1/\alpha}$ (see e.g.~\cite[Theorem 4.5.1]{StochasticProcess}), where $\ell_0(\cdot)$ is a slowly-varying function that is a priori different from $\ell(\cdot)$ in \eqref{eq:definition_alpha_stable_service_distribution}. Moreover, in \eqref{eq:drift_lower_bound_computations} below we show that the highest order contribution to the drift component $R_n(kt)$ is $\Pi(kt)^2/(2n) = \OP(k^2/n)$, all the other terms being negligible. In the process $\mathcal N(\cdot)$ both a drift and a random component appear, so that we must have
\begin{equation}\label{eq:form_scaling_constants}%
\ell_0(k) k^{1/\alpha} = k^2/n.
\end{equation}%
Equivalently,
\begin{equation}\label{eq:form_scaling_constants_computations}%
 \ell_0^{-\frac{\alpha}{2\alpha-1}}(k) k = n^{\frac{\alpha}{2\alpha-1}},
\end{equation}%
where $\ell_0^{-\frac{\alpha}{2\alpha-1}}(\cdot)$ is, by basic properties of slowly varying functions, again slowly varying. On the left-hand side of \eqref{eq:form_scaling_constants_computations} we recognize a regularly-varying function with index $1$. By \cite[Theorem 1.5.12]{bingham1989regular} each regularly-varying function with index $\gamma$ admits an (asymptotic) inverse that is itself regularly varying, with index $1/\gamma$. Therefore, there exists a slowly-varying function $\rho(\cdot)$ so that we must have
\begin{equation}\label{eq:scaling_constant_time}%
k = n^{\frac{\alpha}{2\alpha-1}}\rho(n^{\frac{\alpha}{2\alpha-1}}).
\end{equation}%
Any sequence $(k(n))_{n\geq1}$ that satisfies condition \eqref{eq:scaling_constant_time} is suitable for our purposes, so that we simply take $k(n)=n^{\frac{\alpha}{2\alpha-1}} \ell_1(n)$, where $\ell_1(n) = \rho(n^{\frac{\alpha}{2\alpha-1}})$. Note that $n\mapsto\ell_1(n)$ is again slowly varying. Therefore, the rescaled time parameter is defined as
\begin{equation}%
\tau_n(t) := t n^{\frac{\alpha}{2\alpha-1}}\ell_1(n).
\end{equation}%
We shall denote the time scaling factor by $\tau_n(1) = n^{\frac{\alpha}{\alpha-1}}\ell_1(n)$.
In order to obtain the space-scaling sequence $(s_n)_{n\geq1}$, it is enough to insert $k=n^{\frac{\alpha}{2\alpha-1}} \ell_1(n)$  into $f(k) := k^2/n$. Therefore, we define $s_n$ as
\begin{equation}\label{eq:scaling_constant_space}%
s_n = \big((n^{\frac{\alpha}{2\alpha-1}}\ell_1(n))^2/n\big)^{-1} = \ell_1^{-2}(n)n^{-\frac{1}{2\alpha-1}}=:\ell_2(n) n^{-\frac{1}{2\alpha-1}},
\end{equation}%
%
where $\ell_2(n) = \ell_1^{-2}(n)$ is again slowly varying.
\section{Proof of Theorem \ref{th:scaling_limit_queue_length}}\label{sec:proof_main_theorem}
Rewriting equation \eqref{eq:definition_pre_reflection_process} using \eqref{eq:definition_arrival_process} gives
\begin{align}\label{eq:definition_pre_reflection_process_rewriting}%
N_n(t) &\sr{\mathrm d}{=} N_n(0) + (A^m_n(t) - t/\E[S]) + (B_n(t) /\E[S] - \sigma(B_n(t))\nnl
&=N_n(0) + (\Pi(t) - t/\E[S]) + (B_n(t) /\E[S] - \sigma(B_n(t))) -R_n(t).
\end{align}%
For simplicity, we introduce the scaled version of the arrival and service processes, and of the busy time, as 
\begin{align}%
\mathbf \Pi_n(t) &= n^{-\frac{1}{2\alpha-1}}\ell_2(n)(\Pi(\tau_n(t) )- \tau_n(t)/\E[S]), \nnl
\mathbf R_n(t) &= n^{-\frac{1}{2\alpha-1}}\ell_2(n)R_n(\tau_n(t)),\nnl
\bm \sigma_n(t) &= n^{-\frac{1}{2\alpha-1}}\ell_2(n)(\tau_n(t)/\E[S]-\sigma(\tau_n(t))), \nnl
\hat{\mathbf{ B}}_n(t) &= B_n(\tau_n(t))/\tau_n(1).
\end{align}%
Assume $N_n(0) = q_0 n^{\frac{1}{2\alpha-1}}\ell_1(n)$, for some $q_0\geq0$. After rescaling, equation \eqref{eq:definition_pre_reflection_process_rewriting} is then
\begin{align}\label{eq:pre_reflection_queue_split}%
\mathbf N_n(\tau_n(t)) &=  q_0 + \mathbf \Pi_n(t) + \bm \sigma_n (\hat{\mathbf{B}}_n(t)) - \mathbf R_n(t).
\end{align}%
The proof of Theorem \ref{th:scaling_limit_queue_length} proceeds as follows. First, the term $\mathbf \Pi_n(\cdot)$ is shown to be negligible in the limit. Second, $\bm \sigma_n(\cdot)$ converges to an $\alpha$-stable motion by \eqref{eq:definition_alpha_stable_service_distribution} and Theorem \ref{th:FCLT_equivalence}. Third, $\mathbf R_n(\cdot)$ is shown to converge to the parabolic drift $-\lambda^2/2 t^2$ in Section \ref{sec:drift_limit}. Finally, $\hat{\mathbf{B}}_n(\cdot)$ is shown to converge to the identity function. All these results are then pieced together in Section \ref{sec:proof_main_theorem_wrap_up}. Convergence of the above processes is proven in $\mathcal D([0,T])$ for a fixed $T>0$. Since $T$ is arbitrary, this implies convergence in $\mathcal D([0,\infty])$ by \cite[Lemma 3, p.174]{billingsley1999convergence}.
\begin{lemma}\label{lem:poisson_term_is_negligible}%
As $n\rightarrow\infty$,
\begin{equation}%
\sup_{t\leq T}\vert\mathbf \Pi(\tau_n(t))\vert\sr{\mathbb P}{\rightarrow}0.
\end{equation}%
\end{lemma}%
\begin{proof}%
By the FCLT for the Poisson process,
\begin{equation}%
\frac{\Pi(\tau_n(\cdot))-\tau_n(\cdot)\lambda}{\sqrt{\tau_n(1)}}\sr{\mathrm d}{\rightarrow} B(\cdot), \qquad \text{in}~(\mathcal D, U),
\end{equation}%
where $B(\cdot)$ is a standard Brownian motion, since $1/\E[S]=\lambda$ by the heavy-traffic assumption. By the Skorokhod Representation Theorem, this implies that 
\begin{equation}
\sup_{t\leq T}\Big\vert\frac{\Pi(\tau_n(t))-\tau_n(\cdot)/\E[S]}{\sqrt{\tau_n(1)}} - B(t) \Big\vert \sr{\mathbb P}{\rightarrow}0.
\end{equation}%
Moreover, for any $C>0$ and  $n$ large enough,
\begin{equation}%
C\sqrt{\tau_n(1)} = C n^{\frac{\alpha}{2(2\alpha-1)}}\ell_1^{1/2}(n) \leq n^{1/(2\alpha-1)}\ell_2^{-1}(n),
\end{equation}%
so that $k_n:=n^{1/(2\alpha-1)}\ell_2/\sqrt{\tau_n} \rightarrow\infty$ and
\begin{equation}\label{eq:poisson_process_converges_to_zero}%
\sup_{t\leq T}\Big\vert\frac{\Pi(\tau_n(t))-\tau_n(\cdot)/\E[S]}{n^{1/(2\alpha-1)}\ell_2^{-1}(n)}\Big\vert \leq \frac{1}{k_n} \sup_{t\leq T}\Big\vert\frac{\Pi(\tau_n(t))-\tau_n(t)/\E[S]}{\sqrt{\tau_n(1)}}-B(t)\Big\vert + \sup_{t\leq T} \Big\vert\frac{B(t)}{k_n}\Big\vert.
\end{equation}%
Since the right-hand side of \eqref{eq:poisson_process_converges_to_zero} converges in probability to zero as $n\rightarrow\infty$, the claim follows.
\end{proof}%
Next, we show convergence of the rescaled service process $\sigma(\cdot)$ to an $\alpha$-stable motion:
\begin{lemma}[Stable limit]\label{lem:stable_limit}%
As $n\rightarrow\infty$,
\begin{equation}\label{eq:stable_limit}%
\bm \sigma_n(\cdot)\sr{\mathrm d}{\rightarrow}s_{\alpha}\mathcal S(\cdot)\qquad \mathrm{in}~(\mathcal D, M_1),
\end{equation}%
where $s_{\alpha} = \E[S]^{-(\alpha+{1})/\alpha}$ and $\mathcal S(\cdot)$ is a spectrally positive $\alpha$-stable motion. 
\end{lemma}%
\begin{proof}%
By classical results, the rescaled partial sums of $(S_i)_{i\geq1}$ converge to a spectrally positive $\alpha$-stable motion, see e.g. \cite{jacod2003limit} and \cite[Theorem 4.5.3]{StochasticProcess}. In particular \eqref{eq:FCLT_equivalence_partial_sum_convergence_to_stable} is satisfied. Theorem \ref{th:FCLT_equivalence} implies  \eqref{eq:FCLT_equivalence_renewal_convergence_to_stable}, that is
\begin{equation}%
\bm{\sigma}_n(\cdot) \sr{\mathrm{d}}{\rightarrow} \frac{1}{\E[S]}\mathcal S\Big(\frac{\cdot}{\E[S]}\Big)\qquad\mathrm{in}~(\mathcal D, M_1).
\end{equation}%
By standard properties of stable motion $(\mathcal S(c t))_{t\geq0} \sr{\mathrm d}{=} (c^{1/\alpha} \mathcal S(t))_{t\geq0}$ for $c>0$, so that the claim follows.
\end{proof}%
Although our results do not directly hold for $\alpha=2$ (finite variance case), it is still possible to substitute $\alpha=2$ in the formulas that we obtain, and what is obtained should be consistent with the previously found results for the finite-variance case. This is true, for example, for the coefficient of the stable motion in \eqref{eq:stable_limit}. Indeed, in \cite[Theorem 1]{bet2014heavy} it is proven that if $\E[S^2]=1$, the standard deviation of the limiting Brownian motion is $\lambda^{3/2}=\E[S]^{-3/2}$.

\subsection{Drift limit}\label{sec:drift_limit}
The most difficult task in proving Theorem \ref{th:scaling_limit_queue_length} is to deal with the complicated drift $\mathbf R_n(\cdot)$ in \eqref{eq:pre_reflection_queue_split}. We will prove the following result:
\begin{lemma}[Drift limit]\label{lem:drift_convergence}%
Under the same assumptions as in Theorem \ref{th:scaling_limit_queue_length}, as $n\rightarrow\infty$,
\begin{equation}%
\sup_{t\leq T}\big\vert\mathbf R _n(t) - \frac{\lambda^2}{2}t^2\big\vert \sr{\mathbb P}{\rightarrow}0.
\end{equation}%
\end{lemma}%
The proof will use upper and lower bounds for a distributionally equivalent characterization of $R_n(\cdot)$. First, note that the probability of extracting a mark that has already appeared is $D_n(i-1)/n$, where $D_n(i)$ denotes the number of \emph{different} marks seen up to the  $i$-th arrival epoch in $\Pi(\cdot)$. Therefore, the thinning procedure can be represented by a Bernoulli random variable with parameter $D_n(i-1)/n$. Since at time $t$ there have been a total of $\Pi(t)$ points, we have
\begin{equation}\label{eq:drift_definition}%
R_n(t) \sr{\mathrm d}{=} \sum_{i\leq \Pi(t)} \mathds 1_{\{U_i\leq\frac{D_n(i-1)}{n}\}},
\end{equation}%
where $(U_i)_{i\geq1}$ are uniformly distributed (on $[0,1]$) random variables, independent of all other randomness, and $\mathds 1_{\{U_i\leq x\}}$ is distributed as a Bernoulli random variable with parameter $x$. Moreover, $D_n(i)$ can be written as
\begin{equation}%
D_n(i) = i - Z_{i},
\end{equation}%
where $Z_{i}$ is the number of \emph{repeated} marks seen up to the time of the $i$-th arrival. In other words we have the crucial relation
\begin{equation}
D_n(i) \sr{\mathrm d}{=} i - R_n(\Pi^{-1}(i)),
\end{equation}%
where $\Pi^{-1}(i)$ is the arrival time of the $i$-th customer.
\begin{figure}[!hbt]%
\centering
	\includegraphics{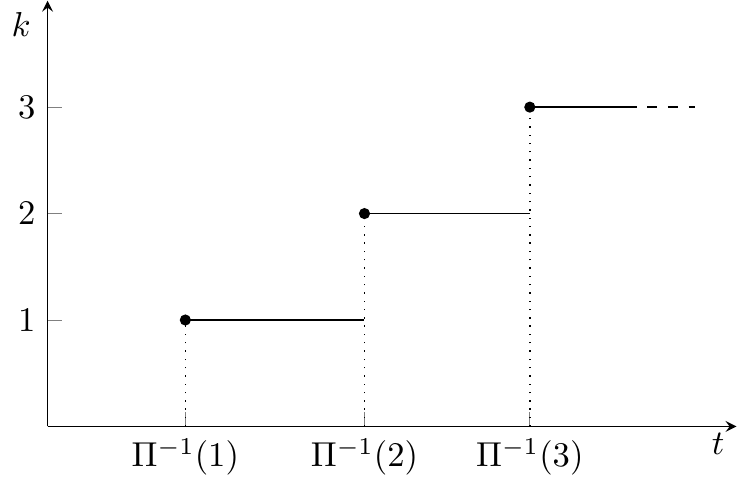}%
\caption{A sample path of the process $\Pi(\cdot)$.}
\end{figure}%

Exploiting these ideas, we can recursively construct a process $(\R_n(k))_{k\geq0}$
 with ${\R}_n(0) := 0$ and
\begin{align}\label{eq:recursive_drift_definition}%
{\R}_n(k) := \sum_{i=1}^k \mathds 1 _{\{ U_i \leq \frac{i-1- \R_n(i-1)}{n} \}},\qquad k\geq1.
\end{align}%
Then
\begin{equation}\label{eq:drift_through_recursive_drift_definition}%
R_n(t) \sr{\mathrm d}{=} \R_n(\Pi(t)).
\end{equation}%
As already mentioned, the processes $R_n(\cdot)$ and $\Pi(\cdot)$ are \emph{not} independent. The distributional equality \eqref{eq:drift_through_recursive_drift_definition} reveals the dependency of $R_n(\cdot)$ on the process $\Pi(\cdot)$.

 The next step is to construct an upper and a lower bound on ${\R}_n(k)$. Since ${\R}_n(k) \geq 0$, the upper bound is trivially
\begin{align}%
\mathds 1 _{\{ U_i \leq (i-1- \R_n(i-1))/n \}} \leq \mathds 1 _{\big\{ U_i \leq \frac{ i-1}{n} \big\}},
\end{align}%
so that, almost surely, 
\begin{align}\label{eq:upper_bound_drift}%
 \R _n (k) \leq  \R^{(\mathrm{up})} _n (k):= \sum_{i=1}^k \mathds 1 _{\big\{ U_i \leq\frac{i-1}{n} \big\}}.
\end{align}%
The lower bound is more involved. By \eqref{eq:upper_bound_drift},
\begin{align}%
\mathds 1 _{\{ U_i \leq \frac{i-1- \R_n(i-1)}{n} \}} \geq \mathds 1_ {\big\{U_i \leq \frac{i-1 -  \R^{(\mathrm{up})} _n(i-1)}{n}\big\}} ,
\end{align}%
so that
\begin{align}%
 \R _n (k)\geq \R^{(\mathrm{low})} _n (k):=\sum_{i=1}^k \mathds 1_ {\big\{U_i \leq\frac{ i-1 -  \R^{(\mathrm{up})} _n(i-1)}{n}\big\}} . 
\end{align}%
We have then constructed a coupling such that \emph{for all} $t\geq 0$,
\begin{equation}\label{eq:two_sided_bound_drift}%
 R^{(\mathrm{low})} _n (t)\preceq R_n(t) \preceq  R^{(\mathrm{up})} _n (t),
\end{equation}%
where $R^{(\mathrm{low})} _n (t):= \tilde R_n^{(\mathrm{low})}(\Pi(t))$ and $R^{(\mathrm{up})} _n(t):= \tilde R_n^{(\mathrm{up})}(\Pi(t))$. For the next and last step we prove uniform convergence of the upper and lower bounds to the same limit. This is done in the following two sections.

%
%
%
%
\subsubsection{Upper bound}
Define the quantity to be estimated as
\begin{align}\label{eq:drift_lower_bound_beginning}%
U_n(T):=\sup_{t\leq T}\Big\vert n^{-1/(2\alpha-1)}\ell_2(n)R^{(\mathrm{up})} _n (\tau_n(t))-\frac{\lambda^2}{2}t^2 \Big\vert,
\end{align}%
We will prove the following:
\begin{lemma}[Upper bound converges to zero]\label{lem:drift_upper_bound}%
Under the assumptions of Theorem \ref{th:scaling_limit_queue_length}, as $n\rightarrow\infty$,
\begin{equation}%
U_n(T)\sr{\mathbb P}{\rightarrow} 0,
\end{equation}%
for every fixed $T>0$.
\end{lemma}%
\begin{proof}
The absolute value in \eqref{eq:drift_lower_bound_beginning} can be split as
\begin{align}\label{eq:drift_lower_bound_computations}%
U_n(T)&\leq\Big\vert n^{-\frac{1}{2\alpha-1}}\ell_2(n)\sum_{i\leq \Pi(\tau_n(t))}\Big( \mathds 1 _{\{ U_i \leq \frac{i-1}{n} \}} -\frac{i-1}{n}\Big)\Big\vert + \Big\vert n^{-\frac{1}{2\alpha-1}}\ell_2(n)\sum_{i\leq \Pi(\tau_n(t))}\frac{i-1}{n} -\frac{\lambda^2}{2}t^2 \Big\vert\nnl
&\leq \Big\vert n^{-\frac{1}{2\alpha-1}}\ell_2(n)\sum_{i\leq\Pi(\tau_n(t))}\Big( \mathds 1 _{\{ U_i \leq \frac{i-1}{n} \}} -\frac{i-1}{n}\Big)\Big\vert + \Big\vert \frac{\Pi(\tau_n(t))^2}{2n^{2\alpha/(2\alpha-1)}\ell_2^{-1}(n)}-\frac{\lambda^2}{2}t^2 \Big\vert +\varepsilon_n,
\end{align}%
where $\varepsilon_n = \vert\Pi(\tau_n(t))/2n\vert$ is an error term. By the functional strong LLN for the Poisson process
\begin{equation}\label{eq:functional_LLN_applied_to_arrival_process}%
\frac{\Pi(t n^{\alpha/(2\alpha-1)}\ell_1(n))}{n^{{\alpha}/{(2\alpha-1)}}\ell_2^{-1/2}(n)}\sr{\mathrm{a.s.}}{\rightarrow} \lambda t,\qquad \text{in}~(\mathcal D,U).
\end{equation}%
It is worth noting that we have made explicit use of the specific form of the scaling functions $\ell_1(\cdot)$ and $\ell_2(\cdot)$ as determined in \eqref{eq:form_scaling_constants} and below. More specifically, by definition $\ell_1^{-2}(n) = \ell_2(n)$.
Moreover, the functional  $x\mapsto x^2$ from $\mathcal D([0,T])$ to itself is almost surely continuous in $f(t) = \lambda t$ in the uniform topology. 
This implies that the second and third terms in \eqref{eq:drift_lower_bound_computations} converge to zero uniformly for $t\leq T$ as $n\rightarrow\infty$.

By the LLN for the Poisson process  $\Pi(s) \leq (\lambda + \varepsilon)s$ w.h.p.~for $s = O(n^{\alpha/(2\alpha-1)})$. 
The sum in the first term in \eqref{eq:drift_lower_bound_computations} can then be bounded on the event $\{\Pi(s)\leq (\lambda + \varepsilon)s\}$ as
\begin{align}\label{eq:drift_lower_bound_martingale_sup}%
\sup_{s\leq \tau_n(T)}&\Big\vert\sum_{i\leq \Pi(s)}\Big( \mathds 1 _{\{ U_i \leq \frac{i-1}{n} \}} -\frac{i-1}{n}\Big)\Big\vert
&\leq \sup_{s\leq (\lambda+\varepsilon)\tau_n(T)}\Big\vert\sum_{i\leq \lfloor s\rfloor}\Big( \mathds 1 _{\{ U_i \leq \frac{i-1}{n} \}} -\frac{i-1}{n}\Big)\Big\vert
\end{align}%
This can be recognized as the supremum of a martingale.  In the following and future computations we shall denote $\bar T := T (\lambda + \varepsilon)$. Then,  an application of Doob's martingale $L^2$ inequality \cite[Theorem 11.2]{klenke2008probability} gives
\begin{align}%
\mathbb P\Big( \sup_{s\leq \bar T n^{\alpha/(2\alpha-1)}\ell_1(n)}\Big\vert\sum_{i\leq \lfloor s\rfloor}&\Big( \mathds 1 _{\{ U_i \leq \frac{i-1}{n} \}} - \frac{i-1}{n}\Big)\Big\vert \geq \varepsilon n^{\frac{1}{2\alpha-1}}\ell_2^{-1}(n)\Big)\nnl
&\leq \sum_{i\leq \bar T n^{\frac{\alpha}{2\alpha-1}}\ell_1(n)}\frac{\mathbb E[( \mathds 1 _{\{ U_i \leq \frac{i-1}{n} \}} - \frac{i-1}{n})^2]}{\varepsilon^2n^{\frac{2}{2\alpha-1}}\ell_2^{-2}(n)}\nnl
&=  \frac{1}{\varepsilon^2 n^{\frac{2}{2\alpha-1}}\ell_2^{-2}(n)}\sum_{i\leq \bar T n^{\frac{\alpha}{2\alpha-1}}\ell_1(n)-1}\Big(\frac{i}{n}-\frac{i^2}{n^2}\Big)\nnl
&\leq\frac{\bar T ^2 n^{\frac{2\alpha}{2\alpha-1}}\ell_1^2(n)}{\varepsilon^2 n^{\frac{2\alpha+1}{2\alpha-1}}\ell_2^{-2}(n)} = O(n^{-\frac{1}{2\alpha-1}}\ell_2(n)),
\end{align}%
and this implies that the right-hand side of \eqref{eq:drift_lower_bound_martingale_sup} is $\oP(n^{1/(2\alpha-1)}\ell_2^{-1}(n))$.
\end{proof}
\subsubsection{Lower bound}
By \eqref{eq:two_sided_bound_drift} we also have 
\begin{equation}%
 R_n(t) \succeq R_n^{(\mathrm{low})}=\sum_{i=1}^{\Pi(t)} \mathds 1_ {\{U_i \leq (i-1 - \tilde R^{(\mathrm{up})} _n(i-1))/n\}}.
\end{equation}%
%
Consequently, we now estimate
\begin{align}\label{eq:drift_upper_bound_beginning}%
L_n(T) :=\sup_{t\leq T}\Big\vert  n^{-1/(2\alpha-1)}\ell_2(n)R^{(\mathrm{low})} _n (\tau_n(t)) - \frac{\lambda^2}{2}t^2 \Big\vert.
\end{align}%
\begin{lemma}[Lower bound converges to zero]\label{lem:drift_lower_bound}%
Under the same assumptions as in Theorem \ref{th:scaling_limit_queue_length}, as $n\rightarrow\infty$,
\begin{equation}%
L_n(T)\sr{\mathbb P}{\rightarrow} 0,
\end{equation}%
for every fixed $T>0$.
\end{lemma}%
\begin{proof}
Similarly as before, conditioned on the event $\{\Pi(s)\leq (\lambda + \varepsilon)s\}$,
\begin{align}\label{eq:drift_upper_bound_computations}%
L_n(T) &\leq \sup_{s\leq\tau_n(\bar T)}\Big\vert n^{-\frac{1}{2\alpha-1}}\ell_2(n)\sum_{i\leq \lfloor s \rfloor}\Big(\mathds 1_ {\{U_i \leq \frac{i-1 - \tilde R^{(\mathrm{up})} _n(i-1)}{n}\}} -\frac{i-1 - \tilde R^{(\mathrm{up})} _n(i-1)}{n}\Big)\Big\vert\\
&\quad + \sup_{t\leq T}\Big\vert n^{-\frac{1}{2\alpha-1}}\ell_2(n)\sum_{i\leq \Pi(\tau_n(t))}\frac{i-1}{n}-\frac{\lambda^2}{2} t^2\Big\vert + \sup_{s\leq \tau_n(\bar T)}\Big\vert n^{-\frac{1}{2\alpha-1}}\ell_2(n)\sum_{i\leq \lfloor s\rfloor}\frac{\tilde R^{(\mathrm{up})} _n(i-1)}{n}\Big\vert\notag.
\end{align}%

The first term in \eqref{eq:drift_upper_bound_computations} can also be bounded as before, since it is again the supremum of a martingale. Indeed, denote $Y_n(i):= (i-1 - \tilde R^{(\mathrm{up})} _n(i-1))/n$ for convenience. By Doob's martingale inequality
\begin{align}\label{eq:drift_upper_bound_martingale_part_computations}%
\varepsilon^2 n^{\frac{2}{2\alpha-1}}&\ell_2^{-2}(n)\mathbb P\Big(\sup_{s\leq \tau_n(\bar T)}\Big\vert\sum_{i\leq \lfloor s\rfloor}\big(\mathds 1_{\{U_i\leq Y_n(i)\}} -Y_n(i)\big)\Big\vert\geq \varepsilon n^{\frac{1}{2\alpha-1}}\ell_2^{-1}(n)\Big) \nnl
&\leq \mathbb E\Big [ \Big(\sum_{i\leq \tau_n(\bar T)}\mathds 1_{\{U_i\leq Y_n(i)\}} -Y_n(i)\Big)^2\Big]= \sum_{i\leq \tau_n(\bar T)}\mathbb E\Big[ \big( \mathds 1_{\{U_i\leq Y_n(i)\}} -Y_n(i)\big)^2\Big].
\end{align}%
%
%
Since the variance of a Bernoulli random variable with parameter $p$ is $p(1-p)$, we get
\begin{align}%
\mathbb E\Big[ \big( \mathds 1_{\{U_i\leq Y_n(i)\}} - Y_n(i)\big)^2\Big] = \mathbb E\Big[Y_n(i) - Y_n(i)^2\Big]\leq \mathbb E[Y_n(i)]\leq \frac{i}{n}.
\end{align}%
This implies that
\begin{equation}%
\sup_{i\leq \bar T n ^{\frac{\alpha}{2\alpha-1}}\ell_1(n)}\mathbb E\Big[ \big( \mathds 1_{\{U_i\leq Y_n(i)\}} - Y_n(i)\big)^2\Big] \leq \bar T n^{\frac{1-\alpha}{2\alpha-1}}\ell_1(n).
\end{equation}%
In particular,
\begin{equation}%
\sum_{i\leq \tau_n(\bar T)}\mathbb E\Big[ \big( \mathds 1_{\{U_i\leq Y_n(i)\}} - Y_n(i)\big)^2\Big] \leq \tau_n(\bar T) \bar T n^{\frac{1-\alpha}{2\alpha-1}}\ell_1^2(n)=\bar T
^2 n^{\frac{1}{2\alpha-1}} \ell_1^2(n) = o(n^{\frac{2}{2\alpha-1}}).
\end{equation}%

The second term in \eqref{eq:drift_upper_bound_computations} has been shown to converge in \eqref{eq:drift_lower_bound_computations} and \eqref{eq:functional_LLN_applied_to_arrival_process}. 

The third term can be bounded, using that $t\mapsto \tilde{R}_n(t)$ is non-decreasing, as 
\begin{align}%
\sup_{s\leq\tau_n(\bar T)}\Big\vert \sum_{i\leq \lfloor s \rfloor}\frac{\tilde R^{(\mathrm{up})} _n(i-1)}{n}\Big\vert &\leq \bar T n^{\frac{1-\alpha}{2\alpha-1}}\ell_1(n) \tilde R_n^{(\text{up})}(\tau_n(\bar T)).
\end{align}%
Note that $\bar T n^{\frac{1-\alpha}{2\alpha-1}}\ell_1(n)\rightarrow0$ as $n\rightarrow\infty$. Since  $n^{-1/(2\alpha-1)}\ell_2(n)\tilde R_n^{(\text{up})}(\tau_n(\bar T))\Pconv 0$ converges by Lemma \ref{lem:drift_upper_bound},
\begin{align*}%
n^{-1/(2\alpha-1)}\ell_2(n)\sup_{s\leq\tau_n(\bar T)}\Big\vert \sum_{i\leq \lfloor s \rfloor}&\frac{\tilde R^{(\mathrm{up})} _n(i-1)}{n}\Big\vert\leq (\bar T n^{\frac{1-\alpha}{2\alpha-1}}\ell_1(n))n^{-1/(2\alpha-1)}\ell_2(n)\tilde R_n^{(\text{up})}(\tau_n(\bar T)) \sr{\mathbb P}{\rightarrow} 0,
\end{align*}%
as $n\rightarrow\infty$. This concludes the proof of Lemma \ref{lem:drift_lower_bound}.
\end{proof}

\begin{proof}[Proof of Lemma \ref{lem:drift_convergence}] Since
\begin{align*}
\sup_{t\leq T} \vert n^{-\frac{1}{2\alpha-1}}\ell_2(n) R_n(tn^{\alpha/(2\alpha-1)}\ell_1(n)) - \frac{1}{2}t^2\vert &=\sup_{t\leq T} (n^{-\frac{1}{2\alpha-1}}\ell_2(n)R_n(tn^{\alpha/(2\alpha-1)}\ell_1(n)) - \frac{1}{2}t^2)^+ \nnl
&\quad+ \sup_{t\leq T}(n^{-\frac{1}{2\alpha-1}}\ell_2(n)R_n(tn^{\alpha/(2\alpha-1)}\ell_1(n)) - \frac{1}{2}t^2)^-,
\end{align*}%
we get
\begin{align}%
\sup_{t\leq T}\Big\vert n^{-\frac{1}{2\alpha-1}}\ell_2(n)R_n(tn^{\alpha/(2\alpha-1)}\ell_1(n)) - \frac{1}{2}t^2\Big\vert \leq  U_n(T) \vee L_n(T) 
\end{align}%
and both $U_n(T)$ and $L_n(T)$ converge in probability to zero by Lemmas \ref{lem:drift_upper_bound} and \ref{lem:drift_lower_bound}. This completes the proof of Lemma \ref{lem:drift_convergence}.
\end{proof}

\subsection{Proof of Theorem \ref{th:scaling_limit_queue_length}}\label{sec:proof_main_theorem_wrap_up}
For the final step, we prove that the cumulative busy time converges to the identity function.

\begin{lemma}[Cumulative idle time is negligible]%
As $n\rightarrow\infty$, 
\begin{equation}%
\hat{\mathbf{B}}_n(t)\sr{\mathrm d}{\rightarrow} \emph{id},\qquad \text{in}~(\mathcal D, U),
\end{equation}%
where $\emph{id}:\mathbb R^+\mapsto \mathbb R^+$ is the identity function.
\end{lemma}%
\begin{proof}%
Since $B_n(t) = t - I_n(t)$, we can equivalently prove that $I_n(t)= \inf_{0\leq s\leq t}(X_n(s)^-)$ converges uniformly to zero, where $X_n(t)$ is the net-input process defined in \eqref{eq:net_input_process}. By continuity of the map $\psi$ given by $\psi:f(\cdot)\rightarrow \inf_{0\leq s\leq \cdot}(f(s)^-)$, it is sufficient to prove that $X_n(\cdot)$ converges uniformly to zero, when appropriately rescaled. By manipulating \eqref{eq:net_input_process} we immediately get
\begin{align}%
\frac{1}{\tau_n(1)}&\sup_{t\leq T}\vert X_n(\tau_n(t))\vert = \sup_{t\leq T}\Big\vert \frac{A_n(\tau_n(t))}{\tau_n(1)}\frac{1}{A_n(\tau_n(t))}\sum_{i=1}^{A_n(\tau_n(t))}S_i-1\Big\vert \nnl
&\leq \sup_{t\leq T}\Big\vert \frac{A_n(\tau_n(t))}{\tau_n(1)}-\frac{1}{\E[S]}\Big\vert\frac{1}{A_n(\tau_n(t))}\sum_{i=1}^{A_n(\tau_n(t))}S_i  + \sup_{t\leq T}\Big\vert\frac{1}{\E[S]}\frac{1}{A_n(\tau_n(t))}\sum_{i=1}^{A_n(\tau_n(t))}S_i-1\Big\vert .
\end{align}%
Note that $\tau_n(t)\rightarrow\infty$ and $A_n(\tau_n(t))\Pconv\infty$ as $n\rightarrow\infty$. Then the second term converges to zero in probability by the Law of Large Numbers (LLN) and the first one converges to zero by the  LLN for the Poisson process. Indeed, $A_n(\tau_n(t)) = \Pi(\tau_n(t)) - R_n(\tau_n(t))$, so that 
\begin{align}\label{eq:idle_time_is_negligible_middle}%
\sup_{t\leq T}\Big\vert \frac{A_n(\tau_n(t))}{\tau_n(1)}-\frac{1}{\E[S]}\Big\vert &\leq \sup_{t\leq T}\Big\vert \frac{\Pi(\tau_n(t))}{\tau_n(1)}-\frac{1}{\E[S]}\Big\vert + \frac{1}{\tau_n(1)}\sup_{t\leq T}\vert R_n(\tau_n(t))\vert\\
&= \sup_{t\leq T}\Big\vert \frac{\Pi(\tau_n(t))}{\tau_n(1)}-\frac{1}{\E[S]}\Big\vert + \frac{n^{1/(2\alpha-1)}\ell_n^{-1}}{\tau_n(1)}\sup_{t\leq T}\frac{\vert R_n(\tau_n(t))\vert}{n^{1/(2\alpha-1)}\ell_n^{-1}}.\notag
\end{align}%
As shown above in Lemma \ref{lem:drift_convergence}, $n^{-1/(2\alpha-1)}\ell_2(n) R_n(\tau_n(t))$ converges uniformly to $-\lambda^2/2 t^2$, and since $n^{1/(2\alpha-1)}\ell_n^{-1}/ \tau_n(1)\rightarrow 0$, the second term in \eqref{eq:idle_time_is_negligible_middle} is negligible. By the heavy-traffic assumption \eqref{eq:definition_criticality_condition} and the LLN for the Poisson process the first term also converges to zero.
\end{proof}%

We now conclude the proof of Theorem \ref{th:scaling_limit_queue_length} by collecting various results from the previous sections.

First, split the process $\mathbf N_n(\cdot)$ in its martingale and drift components as in \eqref{eq:pre_reflection_queue_split} to get
\begin{align}%
&\mathbf N_n(t) =  q_0 + \mathbf \Pi_n (t) + \bm \sigma_n(\hat{\mathbf B}_n(t)) - \mathbf R_n(t).
\end{align}%
Since $\mathbf \Pi_n(\cdot)$ and $\bm\sigma_n(\cdot)$ are independent, and $\hat{\mathbf B}_n(\cdot)$ and $\mathbf R_n(\cdot)$ converge to deterministic limits in $\mathcal D$, we have
\begin{equation}%
(\mathbf \Pi_n(\cdot), \bm\sigma_n(\cdot), \hat{\mathbf B}_n(\cdot), \mathbf R_n(\cdot)) \sr{\mathrm d}{\rightarrow} (0, s_{\alpha}\mathcal S(\cdot) , \mathrm{id}, \lambda^2/2 t^2),\qquad \text{in}~(\mathcal D^4, M^{\mathrm{W}}_1).
\end{equation}%
This, together with a time-change theorem for processes with discontinuous sample paths (e.g. \cite[Theorem 13.2.3]{StochasticProcess}) implies
\begin{equation}\label{eq:joint_convergence_processes}%
(\mathbf \Pi_n(\cdot), \bm\sigma_n(\hat{\mathbf B}_n(\cdot)), \mathbf R_n(\cdot)) \sr{\mathrm d}{\rightarrow} (0, s_{\alpha}\mathcal S(\cdot), \lambda^2/2 t^2),\qquad \mathrm{in}~(\mathcal D^3, M^{\mathrm{W}}_1).
\end{equation}%
Note that \cite[Theorem 13.2.3]{StochasticProcess} does not hold in general in the finer $J_1$ topology. Since the three limit processes in \eqref{eq:joint_convergence_processes} do not have common discontinuity points, addition is continuous in $(0, 1/\E[S]^{(\alpha+{1})/\alpha}\mathcal S(\cdot), \lambda^2/2 t^2)$ in the $M_1$ topology, so that
\begin{equation}%
\mathbf N_n(t) \sr{\mathrm d}{\rightarrow} q_0 + s_{\alpha}\mathcal S(\cdot) - \frac{\lambda^2}{2} t^2,\qquad \text{in}~(\mathcal D,M_1).
\end{equation}%
The second claim \eqref{eq:scaling_limit_queue_length} follows immediately from the Continuous Mapping Theorem, since the reflection map is Lipschitz continuous in the $M_1$ topology by \cite[Theorem 13.5.1]{StochasticProcess}.
\qed
\section{Discussion}\label{sec:discussion}
We have considered a queueing model in which only a finite number of customers can potentially join the system, also referred to as the  $\DG$ model (see \cite{honnappa2014transitory}). For this model, we have defined a suitable heavy-traffic condition, in which the instantaneous arrival rate is assumed to be equal to the service rate. We have shown that, under the additional assumption that the service times obey a power-law with parameter $\alpha\in(1,2)$, the queue length process converges to an $\alpha$-stable process with negative parabolic drift. To prove this, we have given a novel definition of the arrival process that enabled us to obtain explicit bounds on the limiting drift. 
There is a connection between the queueing model and random graphs. Indeed, we can associate to the queueing process a (rooted) random forest, as follows. Customers are the vertices, the first customer in the system is the root, and when customer $i$ joins the queue during the service of customer $j$, an edge is placed between $i$ and $j$. The queue length process then corresponds to the \emph{exploration} of the random tree constructed as above. The exploration process encodes useful information on the underlying random graph. For example, excursions above past minima are the sizes of the connected components. Therefore, our result should be compared with analogous results for other random graph models \cite{BhaHofLee09b, Jose10}. Surprisingly, in the queueing setting the limiting process is much simpler and more intuitive.  Using continuity arguments, we have proved that this implies that the length of the first busy period converges in distribution to the first excursion of the stable motion with negative drift. 

In this paper we have focused on heavy-tailed service times. Thus, Theorem \ref{th:scaling_limit_queue_length} should be compared with the finite-variance case, where the rescaled queue length process converges to a (reflected) Brownian motion with parabolic drift. {This process has been studied by several authors. In \cite{groeneboom1989brownian,groeneboom2010maximum, janson2010maximum} analytic expressions are derived for the joint density of the maximum and location of the maximum of the process $W(t) = B(t) - ct^2$, where $B(\cdot)$ is a Brownian motion and $c>0$ is a constant, and tail estimates are derived in \cite{hofstad2010critical}. In \cite{Aldo97} it is shown that the length of the excursions of $W(t)$ above its past minima can be ordered. 
%

On the other hand,  very little is known about the (reflected) $\alpha$-stable motion with negative quadratic drift. In particular, there are no explicit formulas for the maximum of the free process, and it is not known whether the excursions above past minima can be ordered.} A striking property of the limiting process that we obtain is that  $\sup_{t\geq0} \phi(a+b\mathcal S(t) - ct^2) = \infty$ almost surely. Indeed, it is well known that for any L\'evy process $X(\cdot)$ with unbounded L\'evy measure
\begin{equation}%
\mathbb P (\forall N\in\mathbb N~\forall T>0~\exists t\geq T: \Delta X(t) \geq N) = 1,
\end{equation}%
where $\Delta X(t) := X(t) - \lim_{s\rightarrow t^-}X(s)$. However, due to the parabolic drift the excursions of $\phi(\mathcal N)(\cdot)$ containing a large jump  become smaller as time passes. This suggests that the excursions of $\phi(\mathcal N)(\cdot)$ can be ordered by their time duration and the largest one is finite. In particular it should be possible to prove analytically that for $q_0>0$ large enough the probability that the first busy period is (one of) the largest is close to one. This presents an interesting direction for future research.

\section*{Acknowledgments}

This work is supported by the NWO Gravitation {\sc Networks} grant 024.002.003. The work of RvdH is further supported by the NWO VICI grant 639.033.806. The work of JvL is further supported by an NWO TOP-GO grant and by an ERC Starting Grant.

\bibliographystyle{abbrv}

\DeclareRobustCommand{\VAN}[3]{#3}

\bibliography{library}
\end{document}